\documentclass[11pt]{amsart} 
\usepackage[left=3.5cm,right=3.5cm,top=3cm,bottom=3cm]{geometry} 
\usepackage{amsfonts,amssymb,amsthm,amsmath,amscd}
\usepackage{mathrsfs}
\usepackage{enumerate}
\usepackage{tikz}
\usepackage{apptools}
\usepackage{mathtools}
\usetikzlibrary{graphs}
\title{$\Delta$-Y Moves and Partial Steiner Systems}

\usepackage{graphicx}

\usepackage[dvipsnames]{xcolor}
\definecolor{OrientalRed}{RGB}{192, 44, 56}
\definecolor{OrientalGreen}{RGB}{80, 112, 71}
\definecolor{OrientalBlue}{RGB}{31, 54, 150}
\definecolor{OrientalYellow}{RGB}{228, 169, 77}
\usepackage[backref=page]{hyperref}
\hypersetup{
    colorlinks  = true,     
    urlcolor    = OrientalRed,   
    linkcolor   = OrientalRed,      
    citecolor   = OrientalBlue,     
}

\usepackage{relsize}

\usepackage{nameref}

\newtheorem{theorem}{Theorem}
\newtheorem{lemma}[theorem]{Lemma}

\newtheorem{proposition}[theorem]{Proposition}
\theoremstyle{definition}
\newtheorem{definition}[theorem]{Definition}

\usepackage{xparse}

\NewDocumentEnvironment{proofof}{m o}
{%
    \IfValueT{#2}{%
        \renewcommand{\qedsymbol}{%
            \ensuremath{\square\;\text{\normalfont\scriptsize(#2)}}%
        }%
    }%
    \proof[Proof of #1]%
}
{%
    \endproof
}

\newcommand{\Fstr}{\mathcal{F}_{\text{str}}}
\newcommand{\Gstr}{\mathcal{G}_{\text{str}}}

\begin{document}

\begin{abstract}
We study graph systems generated by $\Delta$-Y and Y-$\Delta$ moves and prove that, for initial graphs of minimum degree at least $7$, $\Delta$-Y moves alone generate every reachable graph. As a consequence, for $n\ge 8$, the $\Delta$-Y system generated from the complete graph of $n$ vertices is canonically isomorphic as graded directed graphs to the space of partial Steiner systems on $n$ points.
\end{abstract}

\author{Zhichen Zhou}
\address{Department of Mathematics, Stockholm University}
\email{zhichen.zhou@math.su.se}
\date{September 8, 2026}


\maketitle

\section{Introduction}
Let $G$ be a simple graph. A $\Delta$-Y move at a triangle $T\subset G$ replaces the three edges of $T$ by a new degree-$3$ vertex adjacent to exactly the three vertices of $T$ while a Y-$\Delta$ move is the inverse of it. The $\Delta$-Y system $\mathcal{G}(G)$ generated by $G$ consists of all graphs that can be obtain from $G$ via a finite sequence of $\Delta$-Y or Y-$\Delta$ moves.

$\Delta$-Y and Y-$\Delta$ moves have appeared in several contexts in graph theory. They are historically motivated by electric networks, and the need for their simplification motivates the theory of $\Delta$-Y/Y-$\Delta$ reducibility, as studied by \cite{ArchdeaconColbournGitlerProvan2000,Yu2006,GitlerSagols2011,Wagner2015}. This idea is intertwined with topological graph theory in the study of spatial ($3$-dimensional) graph embeddings \cite{HanakiEtAl2011,NikkuniTaniyama2012}, which are related to the Conway-Gordon theorems \cite{ConwayGordon1983}, and further of $4$-dimensional graph embeddings as in \cite{VanDerHolst2006,GeorgakopoulosWinter2024}. The moves have also been used as a proof technique in \cite{AichholzerEtAl2014}.

Partial Steiner (triple) systems, equivalently $3$-uniform linear hypergraphs, are standard objects in design theory, see Chapter 9 of \cite{ColbournRosa1999} and Chapter 12 of \cite{HenningYeo2020}. They have been studied from a variety of perspectives, for example, extremal theory as in \cite{ColbournRosaZnam1993}, embeddings as in \cite{BryantHorsley2009}, and probabilistic theory as in \cite{KostochkaRodl1998}.

Given a degree-$3$ vertex $v\in G$ with neighbours $x,y,z\in G$, there are three conventions to define the Y-$\Delta$ move at $v$:
\begin{enumerate}
    \item \label{item:YDelta1}remove $v$ and add edges $xy,yz,zx$, allowing the resultant graph to have multiple edges, such as in \cite{Yu2006}, \cite{GitlerSagols2011}, and \cite{ArchdeaconColbournGitlerProvan2000};
    \item do the same as \ref{item:YDelta1} but then remove parallel edges, forcing the resultant graph to remain simple, such as in \cite{AichholzerEtAl2014};
    \item \label{item:YDelta3}do the same as \ref{item:YDelta1} but only when the neighbourhood of $v$ forms an independent set, as in \cite{HanakiEtAl2011} and \cite{NikkuniTaniyama2012}.
\end{enumerate}

Our Lemma \ref{lem: strong-move sufficiency} unifies these three conventions in the case where we start with graph $G$ of minimum degree $\delta(G) \ge 7$: the system $\mathcal{G}(G)$ consists of simple graphs even if we adapt the first convention, thus there is no parallel edge to remove if we adapt the second convention, and thus both are equivalent with the third convention. We will define a Y-$\Delta$ move in the second convention and call it a \emph{strong} Y-$\Delta$ move (See Definition \ref{def:strongDeltaY}) if condition in \ref{item:YDelta3} is satisfied.

In this paper, we identify the systems $\mathcal{G}(G)$ arising from $\Delta$-Y moves with partial Steiner systems. We note that the $\Delta$-Y move locally encodes the incidence data of a triangle inside the graph itself. Iterating this move gives a graph-theoretic record of a collection of triangles. The idea of this paper is that, assuming a minimum degree condition $\delta(G) \ge 7$, such record is a faithful one.

\subsection*{Main Results}

A \emph{partial Steiner system} on $n$ points is a collection $\mathcal A \subset \binom{[n]}{3}$ of blocks of distinct triples chosen from $1,\dots,n$ such that no pair of points belongs to two blocks in $\mathcal A$. A collection of partial Steiner systems is naturally a graded directed graph where a vertex is assigned to each system $\mathcal{A}$, and a forward edge is given by adjoining one block, where the grading is given by size $|\mathcal{A}|$. We denote the graded directed graph of all partial Steiner systems on $n$ vertices by $\mathcal{P}(n)$.

Let $K_n$ be the complete graph on $n$ vertices. We show that $\mathcal{G}(K_n)$ naturally carries the structure of a directed graph graded by vertex number and is identified to partial Steiner systems as follows.

\begin{theorem}[partial Steiner system identification]\label{thm: partial Steiner system identification}
For every $n\ge 8$, there is a canonical $S_n$-equivariant isomorphism of
graded directed graphs
\[
    \Phi_n:\mathcal{G}(K_n)\xrightarrow{\sim}\mathcal P(n),
\]
given by
\[
    \Phi_n(H)
    =
    \{\text{neighbourhood of }c: c\in H \text{ is a vertex created by some $\Delta$-Y move}\,\}.
\]
\end{theorem}

Here, the isomorphism is canonical with respect to the labelling of vertices of the initial $K_n$ in the sense that once the vertices $V(K_n)$ are identified with $[n]$, no further choice of labelling is involved. The isomorphism is natural in the sense that it is $S_n$-equivariant.

This theorem is a consequence of the following phenomenon.

\begin{theorem}[forward-move sufficiency]
\label{thm: forward-move sufficiency}
Let $G$ be a graph with $\delta(G)\ge 7$, let $\mathcal{F}(G)$ denote the family of graphs obtained from $G$ via finite sequences of $\Delta$-Y and Y-$\Delta$ moves, and let $\mathcal F'(G)$ denote
the family of graphs obtained from $G$ by finite sequences of $\Delta$-Y
moves. Then
\[
    \mathcal F(G)=\mathcal F'(G).
\]
\end{theorem}

As a consequence, $\mathcal{G}(G)$ can be regarded as a graded directed graph with a vertex assigned to each graph $H$ and a forward edge given by each $\Delta$-Y move, where the grading is by vertex number $|V(H)|$.

The condition $\delta(G) \ge 7$ is sharp: a counterexample at minimum degree $6$ is given by the family generated by $K_7$, the famous Heawood family. See, for example, \cite{GoldbergMattmanNaimi2014}.

Indeed, label the points of $K_7$ by $0,1,\dots, 6$. We perform $\Delta$-Y moves on triangles $012, 034,056$ respectively and obtain three created vertices $a,b,c$. Then, vertex $0$ has degree $3$ with neighbourhood $N(0) = \{a,b,c\}$ being an independent set, and we can perform a Y-$\Delta$ move at $0$. In this way, the original vertex $0$ vanishes, so the resultant configuration cannot be obtained from $\Delta$-Y moves alone.

As a further remark, it is known that $\Delta$-Y moves preserve intrinsic knottedness, whereas Y-$\Delta$ moves do not in general \cite{GoldbergMattmanNaimi2014}. Consequently, by Theorem \ref{thm: forward-move sufficiency}, if $G$ is intrinsically knotted and $\delta(G)\ge 7$, then every graph in $\mathcal F(G)$ is intrinsically knotted.

\subsection*{Declaration of AI Use}
We used ChatGPT 5.5 and 5.6 Plus for the aid of searching the literature, exposition, and computation of concrete examples for inspiration.

\section{\texorpdfstring{$\Delta$-Y Systems}{Delta-Y Systems}}\label{sec:deltaWye}
By a graph, we mean a finite simple undirected graph. We call a $3$-cycle a triangle due to its geometric appeal. Similarly, we call an independent set of size $3$ an empty triangle. For a vertex $x$ in a graph $G$, we use $N(x)$ to denote the neighbourhood of $x$.

\begin{definition}[$\Delta$-Y and Y-$\Delta$ moves] \label{def:deltaY}
Let $G$ be a graph. If there exists a triangle $T \subset G$, we may perform a \emph{$\Delta$-Y move} on $G$ at $T$ by removing edges of $T$ and adding a vertex $x$ to $G$ adjacent to exactly the vertices of $T$. We denote the resultant graph by $\Delta(G,T)$. This induces a partial graph embedding, that is an injective graph homomorphism with domain being a subgraph of $G$, 
$$
\iota_T:G \rightharpoonup \Delta(G,T)
$$
with domain $G \setminus E(T)$, the graph obtained by deleting edges of $T$. 

Conversely, if there exists a degree-$3$ vertex $x\in G$, we may perform a \emph{Y-$\Delta$ move} on $G$ at $x$ by removing $x$ and adding edges such that neighbours of $x$ form a triangle. We denote the resultant graph by $Y(G,x)$.
This induces a partial graph embedding 
$$
\iota_x:G \rightharpoonup Y(G,x)
$$
with domain $G \setminus \{x\}$. 
\end{definition}

\begin{definition}[family generated by $\Delta$-Y and Y-$\Delta$ moves]\label{def:familyDeltaY}
Let $G$ be a graph. We define 
\begin{multline*}
\mathcal{F}(G) := \{ \text{graphs }H: H \text{ can be obtained from }G \\
\text{ via a finite sequence of } \Delta \text{-Y or Y-} \Delta \text{ moves} \}. 
\end{multline*}
Similarly, we define
$$
\mathcal{F}'(G) := \{ \text{graphs }H: H \text{ can be obtained from }G 
\text{ via a finite sequence of } \Delta \text{-Y} \text{ moves} \}. 
$$
\end{definition}

${\mathcal{F}}(G)$ is finite because $\Delta$-Y or Y-$\Delta$ move does not increase the number of edges or the number of connected components, and hence both the number of edges and vertices are bounded.

We need to impose a condition on a Y-$\Delta$ move for it to be a two-sided inverse to a $\Delta$-Y move.

\begin{definition}[strong Y-$\Delta$ moves] \label{def:strongDeltaY}
Let $G$ be a graph and $x\in G$ be a vertex of degree $3$. The Y-$\Delta$ move $G\mapsto Y(G,x)$ is called a \emph{strong Y-$\Delta$ move} if the neighbourhood of $x$ form an empty triangle. We use $\mathcal F_{\mathrm{str}}(G)$ to denote the graphs generated by $G$ via a finite sequence of $\Delta$-Y or strong Y-$\Delta$ moves.
\end{definition}

Clearly $\mathcal F_{\mathrm{str}}(G) \subset \mathcal{F}(G)$, so $\mathcal F_{\mathrm{str}}(G)$ is also finite. Our key auxiliary result, Lemma \ref{lem: strong-move sufficiency}, says that $\mathcal F(G) = \Fstr(G)$ if $\delta(G) \ge 7$.

\begin{definition}[graph of graphs]\label{def:G_n}
Given a graph $G$, we define $\mathcal G(G)$ to be the graph with vertex set $V(\mathcal G(G)): = \mathcal F_{\mathrm{}}(G)$ and an edge is assigned between $H \in  V(\mathcal G(G))$ and $K \in V(\mathcal G(G))$ if $K$ is obtainable from $H$ via a $\Delta$-Y or Y-$\Delta$ move.

Given a graph $G$, we define $\Gstr(G)$ to be the simple directed graph with vertex set $V(\Gstr(G)): = \mathcal F_{\mathrm{str}}(G)$ and a forward edge is assigned from $H \in  V(\Gstr(G))$ to $K \in V(\Gstr(G))$ if $K$ is obtainable from $H$ via a $\Delta$-Y move.

If $G = K_n$, we denote $\mathcal G(K_n) = \mathcal{G}(K_n)$ and $\Gstr(K_n) = \Gstr(n)$ in short. We denote $\underline{\Gstr}(G)$ to be the undirected graph associated to $\Gstr (G)$ and $\underline{\Gstr}(n)= \underline{\Gstr}(K_n)$.
\end{definition}

We need to distinguish between original vertices in $G$ and vertices created by $\Delta$-Y moves.

\begin{definition}[created vertex]\label{def:created vertex}
Let $H \in \mathcal F(G)$ and $y \in H$. If there exists a sequence
$$
G=H_0 \xrightharpoonup{M_0} \dots \xrightharpoonup{M_{n-1}} H_n = H
$$
of $\Delta$-Y or Y-$\Delta$ moves sending $G$ to $H$ and some $x \in \operatorname{dom}(M_{n-1}\circ \dots \circ M_0)$ such that $y = M_{n-1}\circ \dots \circ M_0 (x)$, then $y$ is said to be an \emph{original} vertex. Otherwise, $y$ is said to be a \emph{created} vertex.
\end{definition}

The following definition serves to identify a vertex with partial preimage.

\begin{definition}[vertex identification]\label{def:vertex identification}
Let $G,H$ be graphs and $f:G \rightharpoonup H$ be a partial graph embedding. Let $y \in H$. Suppose there exists $x\in G$ such that $f(x) = y$, then we identify $y$ with $x$ and say $y \in G$.
\end{definition}

\begin{definition}[subgraph with vertex restrictions]\label{def:subgraphRes}
Let $G,H$ be graphs and $x_1, \dots, x_n \in H$. We say $(H;x_1,\dots,x_n)$ is a \emph{subgraph} of $G$ or $(H;x_1,\dots,x_n) \subset G$ if there exists a graph embedding $\iota: H \to G$ such that $N(\iota(x_i)) \subset \iota(H)$ for every $i = 1,\dots,n$.
\end{definition}

This definition serves to specify that $x_1,\dots,x_n$ are incident to no extra edges except for those already in $H$.

\begin{lemma}[strong-move sufficiency]\label{lem: strong-move sufficiency}
Let $G$ be a graph with minimum degree $\delta(G)\ge 7$, then $\mathcal F_{\mathrm{str}}(G) = \mathcal{F}(G)$. That is to say, every Y-$\Delta$ move on any $H \in \mathcal F(G)$ is automatically strong.
\end{lemma}

We remark that \cite{GoldbergMattmanNaimi2014} observed by exhaustive calculation that $\mathcal F(G) = \mathcal F_{\mathrm{str}}(G)$ for several concrete graphs $G$ such as $G = K_7$ or $K_{3,3,1,1}$, and this lemma resolves this issue in the regime where $\delta(G) \ge 7$ . We conjecture that the tight condition for this lemma is $\delta(G) \ge 6$.

\begin{proofof}{Lemma \ref{lem: strong-move sufficiency}}[Lemma \ref{lem: strong-move sufficiency}]
Note that $\mathcal F_{\mathrm{str}}(G) = \mathcal{F}(G)$ unless there exists $H \in \Fstr(G)$ that admits a non-strong Y-$\Delta$ move, so that $H$ contains a degree-$3$ vertex $x$ whose neighbourhood is not an empty triangle. That is to say, denoting $$A :=
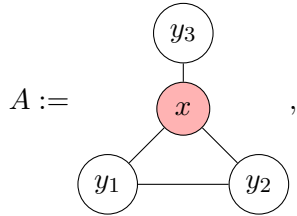
\begin{tikzpicture}[baseline=(current bounding box.center), every node/.style={circle, draw, minimum size=7mm}]
  \node [fill=red!30] (x)  at (0, 0)  {$x$};
  \node (y1) at (-1, -1)  {$y_1$};
  \node (y2) at (1,-1)  {$y_2$};
  \node (y3) at (0,1)  {$y_3$};

  \draw (x)  -- (y1);
  \draw (x)  -- (y2);
  \draw (x)  -- (y3);
  \draw (y1) -- (y2);
\end{tikzpicture},$$ we have $(A;x) \subset H$.
Without loss of generality, assume that $H$ has minimum distance from $G$ in $\underline{\Gstr}(G)$ with $(A;x)\subset H$. We note that $G \neq H$ since $\delta(G) > 3$.

Let $H' \in N_{\underline{\Gstr}(G)}(H)$ such that $d_{\underline{\Gstr}}(G,H') = d_{\underline{\Gstr}}(G,H)-1$, then by minimality $(A;x) \not \subset H'$. This implies either a vertex or an edge is created by a $\Delta$-Y move or strong Y-$\Delta$ move from $H'$ to $H$. We will discuss which vertices or edges are created. 

We need the following claims first:

\noindent \textbf{Claim 1.} For every $L \in \Fstr (G)$ with $d_{\underline{\Gstr}}(G,L) < d_{\underline{\Gstr}}(G,H)$ and $v\in L$ original, we have $\deg v \ge 4$.

\noindent \textbf{Claim 2.} Let $L \in \Fstr (G)$ with $d_{\underline{\Gstr}}(G,L) < d_{\underline{\Gstr}}(G,H)$. Then, the neighbourhood of any created vertex is invariant under $\Delta$-Y or Y-$\Delta$ moves.

\begin{proofof}{Claim 1}[Claim 1]
    Suppose otherwise and let $\tilde{L}$ be a counterexample with minimum $\underline{\Gstr}$-distance from $G$. Then, $\tilde{L}$ has an original vertex $v$ of degree $\le 3$, so $\tilde{L} \neq G$. We now claim the following.

\noindent\textbf{Claim 1.1.} For every $J \in \Fstr (G)$ with $d_{\underline{\Gstr}}(G,J) \le d_{\underline{\Gstr}}(G,\tilde{L})$, if $c \in J$ is a created vertex, then $\deg c = 3$. 

\begin{proofof}{Claim 1.1}[Claim 1.1]
There exists a chain of $\Delta$-Y or strong Y-$\Delta$ moves of minimum length $
G=J_0 \xrightharpoonup{M_0} \dots \xrightharpoonup{M_{n-1}} J_n = J
$ from $G$ to $J$. Let $i$ be the minimal index such that $c\in J_i$, then $i > 0$. Then, we have $c \in J_j$ for all $j \ge i$. We denote $d_j := \deg_{J_j} c$ for $j\ge i$, so $d_i =3$. Suppose $d_j \le 2$ for some $j > i$, and without loss of generality let $j$ be the minimal index for such instance. Since the only way to reduce the degree of a vertex is to perform a $\Delta$-Y move on a triangle containing it, we have that two neighbours of $c$ in $J_{j-1}$ are adjacent. Thus, we have $(A;x) \hookrightarrow J_{j-1}$ sending $x \mapsto c$ and $M_{j-1} = \iota_T$ with vertices of $T$ given by $\{c = x, y_1,y_2\}$. This is impossible by minimality of $H$, since $d_{\underline{\Gstr}}(G,J_{j-1}) < d_{\underline{\Gstr}}(G,H)$. \\
For the other direction, suppose $d_j \ge 4$ for some $j > i$, and without loss of generality let $j$ be the minimal index for such instance. Since the only way for $c$ to gain a degree is to perform a Y-$\Delta$ transform at a neighbour of $c$, we have, for $B := $
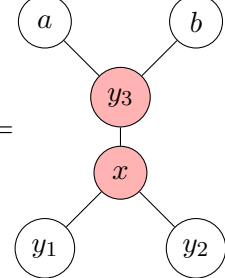
\begin{tikzpicture}[baseline=(current bounding box.center), every node/.style={circle, draw, minimum size=7mm}]
  \node [fill=red!30] (y3)  at (0, 1)  {$y_3$};
  \node [fill=red!30] (x)  at (0, 0)  {$x$};
  \node (y1) at (-1, -1)  {$y_1$};
  \node (y2) at (1,-1)  {$y_2$};
  \node (a) at (-1,2)  {$a$};
  \node (b) at (1,2) {$b$};

  \draw (x)  -- (y1);
  \draw (x)  -- (y2);
  \draw (y3)  -- (a);
  \draw (y3)  -- (b);
  \draw (y3) -- (x);
\end{tikzpicture}, that $(B;x,y_3) \hookrightarrow J_{j-1}$ sending $x \mapsto c$ and $M_{j-1} = \iota_{y_3}$. Note that $d_{\underline{\Gstr}}(G,J_{j-1}) < d_{\underline{\Gstr}}(G,J) \le d_{\underline{\Gstr}}(G,\tilde{L})$ by minimality of chain length, then, by minimality of $\tilde{L}$, as $\deg y_3 <4$, $y_3$ is a created vertex. Then, we take $k$ to be the minimal index such that $(B;c,y_3) \subset J_{k}$, so $k > 0$. Suppose that there is a vertex in $B$ that is created from $M_{k-1}: J_{k-1} \rightharpoonup J_k$. By symmetry, there are two cases to discuss: whether $a$ or $x$ is created from $M_{k-1}$. If $a$ is created, then there is a triangle $T$ in $J_{k-1}$ with $y_3 \in T$. Since after a $\Delta$-Y move at $T$, the vertices of $T$ form an independent set in $J_{k}$, we have that $b \not \in T$. As a consequence, the subgraph generated $T \cup \{b\}$ is a copy of $A$ in $J_{k-1}$, contradicting the minimality of $H$. For the other case if $x$ is created, we have that $y_1,y_2,y_3$ form a triangle in $J_{k-1}$, so that $y_1,y_2,y_3,a$ generate a copy of $A$ in $J_{k-1}$, contradicting the minimality of $H$. 

Now, suppose that no vertex is created by $M_{k-1}$, then $M_{k-1}$ is a Y-$\Delta$ move, since otherwise created edges are incident to created vertices. Suppose one edge is created by $M_{k-1}$. By the degree condition on $x$ and $y_3$, without loss of generality, the only possibility is that $M_{k-1}$ acts on the triangle generated by $\{a,y_3,b\}$. Then, $y_3$ has degree $2$ in $J_{k-1}$, and this is impossible by the above if $y_3$ is a created vertex and by the minimality of $\tilde{L}$ if $y_3$ is original.

Finally, if $c,c' \in J$ are adjacent created vertices, since both have degree $3$, the subgraph generated by $c,c'$ and their neighbours contains a copy of $B$, which contradicts the minimality of $H$ by the above argument.
\end{proofof}

\noindent\textbf{Claim 1.2.} For every $J \in \Fstr (G)$ with $d_{\underline{\Gstr}}(G,J) \le d_{\underline{\Gstr}}(G,\tilde{L})$, if $c \in J$ is a created vertex, then $c$ does not appear in a triangle.

\begin{proofof}{Claim 1.2}[Claim 1.2]
Suppose $c \in T \subset J$ for some triangle $T$. Since $\deg c = 3$ by Claim 1.1, the triangle $T$ together with the other edge of $c$ form a copy of $(A;x)$ in $J$, and this contradicts the minimality of $H$. 
\end{proofof}

\noindent \textbf{Claim 1.3.} For every $J \in \Fstr (G)$ with $d_{\underline{\Gstr}}(G,J) \le d_{\underline{\Gstr}}(G,\tilde{L})$ and $c,c'\in J$ created vertices, we have that $c,c'$ are not adjacent.

\begin{proofof}{Claim 1.3}[Claim 1.3]
By Claim 1.1, both $c,c'$ have degree $3$. By Claim 1.2, $c,c'$ have no common neighbours. Then, the subgraph generated by $c,c'$ and their neighbours contains a copy of $B$, which contradicts the minimality of $H$ by the argument at the end of the proof of Claim 1.1.
\end{proofof}

We go back to the proof of Claim 1. Consider the chain $
G=\tilde{L}_0 \xrightharpoonup{M_0} \dots \xrightharpoonup{M_{n-1}} \tilde{L}_n = \tilde{L}
$ as in the proof of Claim 1.1. We denote $d'_j :=\deg_{\tilde{L}_j}v$ for $0 \le j \le n$. Let $S \subset \{0,\dots,n-1\}$ be the subset of indices $j$ such that $d'_{j+1} < d'_j$. Then $d'_{j} - d'_{j+1} = 1$ and $M_j = \iota_{T_j}$, where $T_j$ is a triangle containing $v$ and two neighbours of $v$ that are original vertices, by Claim 1.2. Similarly, let $U \subset \{0,\dots,n-1\}$ be the subset of indices $j$ such that $d'_{j+1} > d'_j$. Then $d'_{j+1} - d'_{j} = 1$ and $M_j = \iota_{w_j}$, where $w_j$ is a degree-$3$ neighbour of $v$. By minimality of $\tilde{L}$, $w_j$ is a created vertex, so that the neighbours of $w_j$ are original vertices by Claim 1.3. 

Therefore, we have that $$\deg_{\tilde{L}} v = \deg_G v -|S| + |U|$$ and $$0 \le \# \text{original vertex neighbours of }v\in \tilde{L} \le \deg_G v - 2|S| + 2|U|.$$
As a result, $\deg_{\tilde{L}} v \ge \frac{1}{2} \deg_G v \ge \frac{7}{2}$, and this contradicts the choice of $v$ which has degree $\deg_{\tilde L}v \le 3$
\end{proofof}

Although Claims 1.1 to 1.3 above are local claims about the supposed counterexample $\tilde{L}$ for Claim 1, now as Claim 1 is proved, we can upgrade them into a more global version:

\noindent \textbf{Claim 1.1$'$.} For every $L \in \Fstr (G)$ with $d_{\underline{\Gstr}}(G,L) < d_{\underline{\Gstr}}(G,H)$, if $c \in L$ is a created vertex, then $\deg c = 3$.

\noindent \textbf{Claim 1.2$'$.} For every $L \in \Fstr (G)$ with $d_{\underline{\Gstr}}(G,L) < d_{\underline{\Gstr}}(G,H)$, if $c \in L$ is a created vertex, then $c$ does not appear in a triangle.

\noindent \textbf{Claim 1.3$'$.} Let $L \in \Fstr (G)$ with $d_{\underline{\Gstr}}(G,L) < d_{\underline{\Gstr}}(G,H)$. Then, created vertices in $L$ are not adjacent to each other.

\begin{proofof}{Claims 1.1$'$ to 1.3$'$}[Claims 1.1$'$ to 1.3$'$]
Go through the same argument as in the proofs of Claims 1.1 to 1.3, except for replacing ``by minimality of $\tilde{L}$" with Claim 1.
\end{proofof}

\begin{proofof}{Claim 2}[Claim 2]
There are two ways to alter the neighbourhood of $c$: to perform a $\Delta$-Y move on a triangle containing $c$ or a Y-$\Delta$ move on a neighbour of $c$. The former is impossible by Claim 1.2$'$. In the latter case, we have $(B;x,y_3) \subset L$ as in the proof of Claim 1.1, and this is impossible following the same argument above, replacing ``by minimality of $\tilde{L}$" with Claim 1.
\end{proofof}
We return to the main proof of Lemma \ref{lem: strong-move sufficiency}. Recall that $H$ has minimum distance from $G$ such that $(A;x) \subset H$. We will show that such minimality is contradicted.

Suppose some vertex is created from $H'$ to $H$, then the move $M:H' \rightharpoonup H$ must be a $\Delta$-Y move. Since created vertices have an independent neighbourhood, the possible case is where $y_3$ is created. Then we have $(C;x) \subset H'$ or $(D;x) \subset H'$, where \\$C := $ 
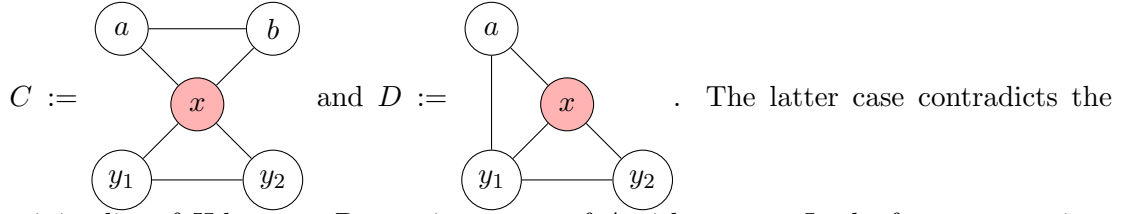
\begin{tikzpicture}[baseline=(current bounding box.center), every node/.style={circle, draw, minimum size=7mm}]
  \node [fill=red!30] (x)  at (0, 0)  {$x$};
  \node (y1) at (-1, -1)  {$y_1$};
  \node (y2) at (1,-1)  {$y_2$};
  \node (a) at (-1,1)  {$a$};
  \node (b) at (1,1) {$b$};

  \draw (x)  -- (y1);
  \draw (x)  -- (y2);
  \draw (x)  -- (a);
  \draw (x)  -- (b);
  \draw (y1) -- (y2);
  \draw (a) -- (b);
\end{tikzpicture}
and $D := $
\begin{tikzpicture}[baseline=(current bounding box.center), every node/.style={circle, draw, minimum size=7mm}]
  \node [fill=red!30] (x)  at (0, 0)  {$x$};
  \node (y1) at (-1, -1)  {$y_1$};
  \node (y2) at (1,-1)  {$y_2$};
  \node (a) at (-1,1)  {$a$};

  \draw (x)  -- (y1);
  \draw (x)  -- (y2);
  \draw (x)  -- (a);
  \draw (y1)  -- (a);
  \draw (y1) -- (y2);
\end{tikzpicture}. The latter case contradicts the minimality of $H$ because $D$ contains a copy of $A$ with $y_3 \mapsto a$. In the former case, since $\deg x \ge 4$, $x$ is an original vertex. Since $\delta(G)\ge 7$ and by the proof of Claim 1, for $\deg_{H'}x = 4$ to happen, we must have that at least three neighbours of $x$ are created vertices. Since created vertices are not adjacent to each other by Claim 1.3$'$, this is impossible. 

Now suppose no vertex is created, then $M$ must be a Y-$\Delta$ move, since any edge being created by a $\Delta$-Y move has an incident vertex being created by the same move.
If three edges are created in forming $(A;x)$, then the edges created must be $\{y_1y_2,y_2y_3,y_3y_1\}$.

In this case, we have $(E;x,a) \subset H'$, where $E := $
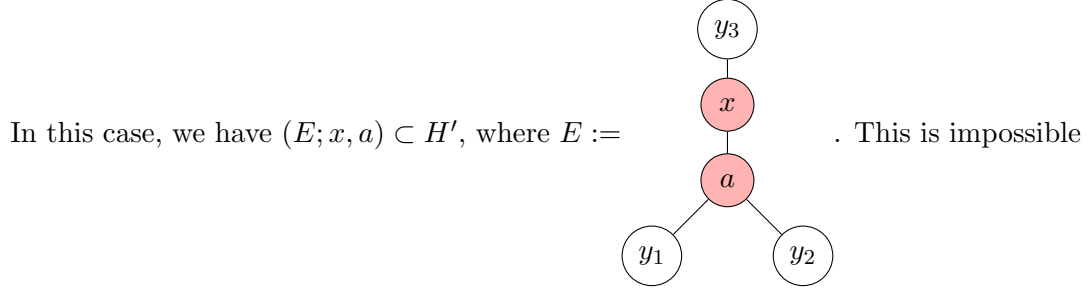
\begin{tikzpicture}[baseline=(current bounding box.center), every node/.style={circle, draw, minimum size=7mm}]
  \node [fill=red!30] (y3)  at (0, 1)  {$x$};
  \node [fill=red!30] (x)  at (0, 0)  {$a$};
  \node (y1) at (-1, -1)  {$y_1$};
  \node (y2) at (1,-1)  {$y_2$};
  \node (a) at (0,2)  {$y_3$};

  \draw (x)  -- (y1);
  \draw (x)  -- (y2);
  \draw (y3)  -- (a);
  \draw (y3) -- (x);
\end{tikzpicture}.  This is impossible since $\delta(H') \ge 3$ by Claim 1 and Claim 1.1$'$, whereas $x$ has degree only $2$.

If two edges are created, we discuss cases. If $y_3 x$ is one of the edges being created, then the other edge is $y_2x$ or $y_1x$ since created edges are adjacent. Assuming without loss of generality that $y_3x$ and $y_1x$ are created, then the move $M:H' \rightharpoonup H$ is a Y-$\Delta$ move because $x$ cannot be created from $M$. Then, for $F := $
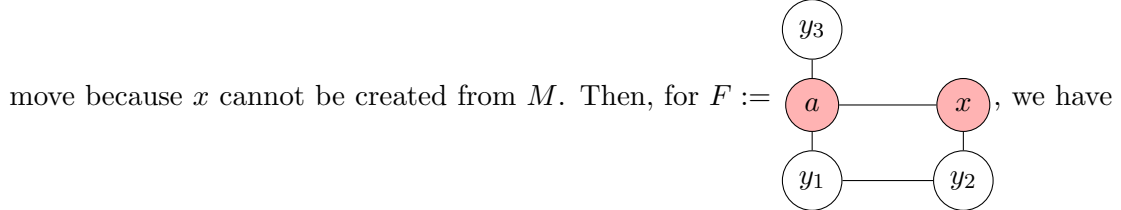
\begin{tikzpicture}[baseline=(current bounding box.center), every node/.style={circle, draw, minimum size=7mm}]
  \node [fill=red!30] (x)  at (1, 0)  {$x$};
  \node (y1) at (-1, -1)  {$y_1$};
  \node (y2) at (1,-1)  {$y_2$};
  \node [fill=red!30] (a) at (-1,0)  {$a$};
  \node (y3) at (-1,1) {$y_3$};

  \draw (x)  -- (y2);
  \draw (x)  -- (a);
  \draw (a)  -- (y1);
  \draw (y1) -- (y2);
  \draw (a) -- (y3);
\end{tikzpicture}, we have $(F;a,x) \subset H'$, and this is again excluded by the degree of $x$. If $y_2x$ and $y_1 x$ are created, then $M$ is a Y-$\Delta$ move creating triangle generated by $\{x,y_1,y_2\}$, and we return to the case where three edges are created. The case is similar if $y_2x$  and $y_1 y_2$ are created.

Suppose only one edge is created. If $y_1y_2$ is created, then for $A_1 := $
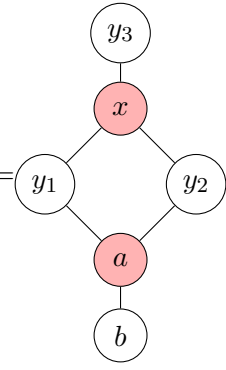
\begin{tikzpicture}[baseline=(current bounding box.center), every node/.style={circle, draw, minimum size=7mm}]
  \node [fill=red!30] (x)  at (0, 0)  {$x$};
  \node (y1) at (-1, -1)  {$y_1$};
  \node (y2) at (1,-1)  {$y_2$};
  \node (y3) at (0,1)  {$y_3$};
  \node [fill=red!30] (a) at (0,-2)  {$a$};
  \node (b) at (0,-3) {$b$};

  \draw (x)  -- (y3);
  \draw (x)  -- (y2);
  \draw (y2)  -- (a);
  \draw (a)  -- (y1);
  \draw (x) -- (y1);
  \draw (a) -- (b);
\end{tikzpicture}
, $A_2 := $ \begin{tikzpicture}[baseline=(current bounding box.center), every node/.style={circle, draw, minimum size=7mm}]
  \node [fill=red!30] (x)  at (0, 0.5)  {$x$};
  \node (y1) at (-1, -1)  {$y_1$};
  \node (y2) at (1,-1)  {$y_2$};
  \node (y3) at (0,-1)  {$y_3$};
  \node [fill=red!30] (a) at (0,-2.5)  {$a$};

  \draw (x)  -- (y3);
  \draw (x)  -- (y2);
  \draw (y2)  -- (a);
  \draw (a)  -- (y1);
  \draw (x) -- (y1);
  \draw (a) -- (y3);
\end{tikzpicture}, and $A_3 := $ \begin{tikzpicture}[baseline=(current bounding box.center), every node/.style={circle, draw, minimum size=7mm}]
  \node [fill=red!30] (x)  at (0, 0)  {$x$};
  \node (y1) at (-1, -1)  {$y_1$};
  \node (y2) at (1,-1)  {$y_2$};
  \node (y3) at (0,1)  {$y_3$};
  \node [fill=red!30] (a) at (0,-2)  {$a$};

  \draw (x)  -- (y3);
  \draw (x)  -- (y2);
  \draw (y2)  -- (a);
  \draw (a)  -- (y1);
  \draw (x) -- (y1);
  \draw (a) -- (x);
\end{tikzpicture},
we have $(A_i;x,a) \subset H'$ for $i=1,2, \text{ or }3$ and $M = \iota_a$ in each case. Case $i=3$ is impossible by minimality of $H$ because $A_3$ contains a copy of $A$ with centre $x \mapsto a$. For case $i=1$ or $2$, since $\deg x = \deg a =3$, both $x$ and $a$ are created vertices. Let $
G=H_0 \xrightharpoonup{M_0} \dots \xrightharpoonup{M_{n-1}} H_n = H
$ be a chain of $\Delta$-Y or strong Y-$\Delta$ moves of minimum length from $G$ to $H$. Then, by symmetry, we may assume that $x$ is created before $a$, say $x$ is created at $M_{j}$ and $a$ is created at $M_k$ for $j<k$. Since the neighbourhood of created vertices is invariant under moves by Claim 2, we have, for $A_4 := $\begin{tikzpicture}[baseline=(current bounding box.center), every node/.style={circle, draw, minimum size=7mm}]
  \node [fill=red!30] (x)  at (0, 0)  {$x$};
  \node (y1) at (-1, -1)  {$y_1$};
  \node (y2) at (1,-1)  {$y_2$};
  \node (y3) at (0,1)  {$y_3$};
  \node (b) at (0,-3) {$b$};

  \draw (x)  -- (y3);
  \draw (x)  -- (y2);
  \draw (y2)  -- (y1);
  \draw (b)  -- (y1);
  \draw (x) -- (y1);
  \draw (y2) -- (b);
\end{tikzpicture}
and $A_5 := $
\begin{tikzpicture}[baseline=(current bounding box.center), every node/.style={circle, draw, minimum size=7mm}]
  \node [fill=red!30] (x)  at (0, 0)  {$x$};
  \node (y1) at (-1, -1)  {$y_1$};
  \node (y2) at (1,-1)  {$y_2$};
  \node (y3) at (0,1.5)  {$y_3$};

  \draw (x)  -- (y3);
  \draw (x)  -- (y2);
  \draw (y2)  -- (y1);
    \draw (y3)  -- (y1);
      \draw (y2)  -- (y3);
  \draw (x) -- (y1);
\end{tikzpicture}, we have $(A_4;x) \subset H_k$ in case $i=1$ and $(A_5;x) \subset H_k$ in case $i=2$. Both cases are rejected by minimality of $H$.\\

If $xy_2$ is created by $M$, then for $B_1 := $
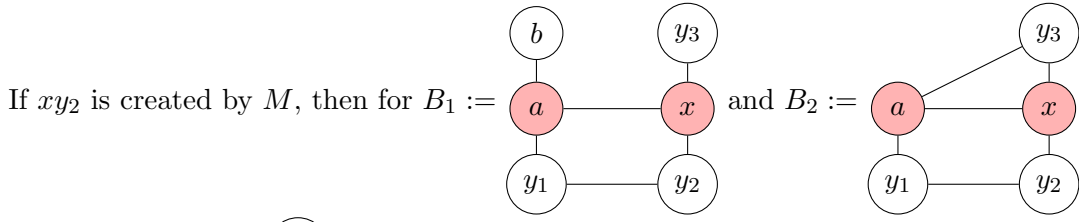
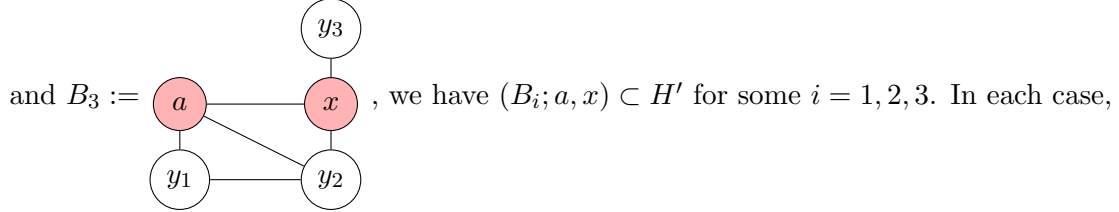
\begin{tikzpicture}[baseline=(current bounding box.center), every node/.style={circle, draw, minimum size=7mm}]
  \node [fill=red!30] (x)  at (1, 0)  {$x$};
  \node (y1) at (-1, -1)  {$y_1$};
  \node (y2) at (1,-1)  {$y_2$};
  \node (y3) at (1,1)  {$y_3$};
  \node [fill=red!30] (a) at (-1,0)  {$a$};
  \node (b) at (-1,1) {$b$};

  \draw (x)  -- (y3);
  \draw (x)  -- (y2);
  \draw (x)  -- (a);
  \draw (a)  -- (y1);
  \draw (y1) -- (y2);
  \draw (a) -- (b);
\end{tikzpicture} and $B_2 := $
\begin{tikzpicture}[baseline=(current bounding box.center), every node/.style={circle, draw, minimum size=7mm}]
  \node [fill=red!30] (x)  at (1, 0)  {$x$};
  \node (y1) at (-1, -1)  {$y_1$};
  \node (y2) at (1,-1)  {$y_2$};
  \node (y3) at (1,1)  {$y_3$};
  \node [fill=red!30] (a) at (-1,0)  {$a$};

  \draw (x)  -- (y3);
  \draw (x)  -- (y2);
  \draw (x)  -- (a);
  \draw (a)  -- (y1);
  \draw (y1) -- (y2);
  \draw (a) -- (y3);
\end{tikzpicture} and $B_3 :=$
\begin{tikzpicture}[baseline=(current bounding box.center), every node/.style={circle, draw, minimum size=7mm}]
  \node [fill=red!30] (x)  at (1, 0)  {$x$};
  \node (y1) at (-1, -1)  {$y_1$};
  \node (y2) at (1,-1)  {$y_2$};
  \node (y3) at (1,1)  {$y_3$};
  \node [fill=red!30] (a) at (-1,0)  {$a$};

  \draw (x)  -- (y3);
  \draw (x)  -- (y2);
  \draw (x)  -- (a);
  \draw (a)  -- (y1);
  \draw (y1) -- (y2);
  \draw (a) -- (y2);
\end{tikzpicture}
  , we have $(B_i;a,x) \subset H'$ for some $i = 1,2,3$. In each case, we have $M = \iota_a$. Since $\deg a = \deg x = 3$, both $a$ and $x$ are created vertices. By Claim 1.3$'$, it is impossible for $a$ to be adjacent to $x$, thus we have a contradiction.\\

If $xy_3$ is created, then for $C_1 := $
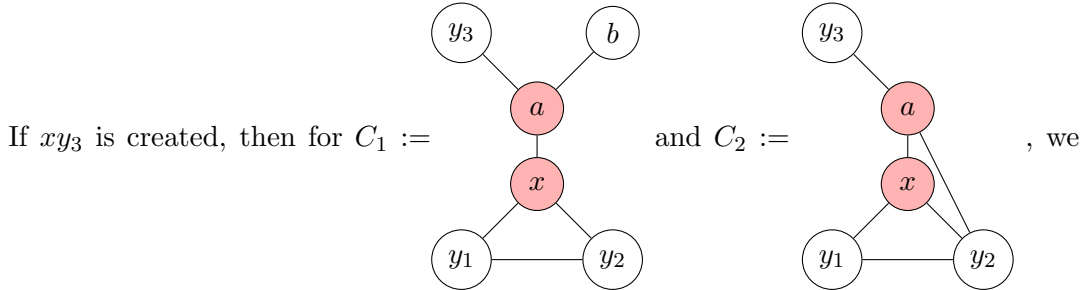
\begin{tikzpicture}[baseline=(current bounding box.center), every node/.style={circle, draw, minimum size=7mm}]
  \node [fill=red!30] (a)  at (0, 1)  {$a$};
  \node [fill=red!30] (x)  at (0, 0)  {$x$};
  \node (y1) at (-1, -1)  {$y_1$};
  \node (y2) at (1,-1)  {$y_2$};
  \node (y3) at (-1,2)  {$y_3$};
  \node (b) at (1,2) {$b$};

  \draw (x)  -- (y1);
  \draw (x)  -- (y2);
  \draw (y3) -- (a);
    \draw (y1) -- (y2);
  \draw (a)  -- (b);
  \draw (a) -- (x);
\end{tikzpicture} and $C_2 := $
\begin{tikzpicture}[baseline=(current bounding box.center), every node/.style={circle, draw, minimum size=7mm}]
  \node [fill=red!30] (a)  at (0, 1)  {$a$};
  \node [fill=red!30] (x)  at (0, 0)  {$x$};
  \node (y1) at (-1, -1)  {$y_1$};
  \node (y2) at (1,-1)  {$y_2$};
  \node (y3) at (-1,2)  {$y_3$};

  \draw (x)  -- (y1);
  \draw (x)  -- (y2);
  \draw (y3) -- (a);
    \draw (y1) -- (y2);
  \draw (a)  -- (y2);
  \draw (a) -- (x);
\end{tikzpicture}
, we have $(C_1;a,x)\subset H'$ or $(C_2;a,x)\subset H'$, and $M = \iota_a$ in either case. Thus, in either case, we have a contradiction to the minimality of $H$.\\

The proof is complete.
\end{proofof}

As the lemma is established, the claims established in the proof of it are in fact true for any graph in $\mathcal F_{\mathrm{str}}(G)$, and this is summarised in Proposition \ref{prop:vertex-type_separation}.

\begin{proposition}[vertex-type separation]\label{prop:vertex-type_separation}
Let $G$ be a graph with $\delta(G)\ge 7$, and let $H\in\mathcal F(G)$.
Then:
\begin{enumerate}[(i)]
    \item \label{prop1-1}every original vertex of $H$ has degree at least $4$;
    \item \label{prop1-2}every created vertex of $H$ has degree $3$;
    \item \label{prop1-3}no created vertex of $H$ lies in a triangle;
    \item \label{prop1-4}no two created vertices of $H$ are adjacent.
    \item \label{prop1-5}the neighbourhood of every created vertex of $H$ is invariant
    under $\Delta$-Y and Y-$\Delta$ moves given that the vertex exists after the move;
\end{enumerate}
\end{proposition}

\begin{proof}
Note that this proposition is the collection of results in Claims 1, 1.1$'$, 1.2$'$, 1.3$'$, and 4 without local assumptions. By Lemma \ref{lem: strong-move sufficiency}, $\mathcal F(G) = \Fstr(G)$. Now, we let $\tilde{L}$ be a counterexample of part \ref{prop1-1} with minimum $\underline{\Gstr}$-distance from $G$. Then, similar to Claim 1.1, we show that part (\ref{prop1-2}) holds for every graph $J \in \Fstr(G)$ with $d_{\underline{\Gstr}}(G,J) \le d_{\underline{\Gstr}}(G,\tilde{L})$. The proof is the same as the proof of Claim 1.1 except we replace ``minimality of $H$" with Lemma \ref{lem: strong-move sufficiency}, as $A$ is not contained in any graph in $\Fstr(G)$ by the lemma. Then, we show parts \ref{prop1-2}, \ref{prop1-3}, \ref{prop1-4} hold for $J$ as in the claims. This will show \ref{prop1-1} in general, provided that we replace ``minimality of $H$" with Lemma \ref{lem: strong-move sufficiency}. Then, as in the proofs of Claims 1.1$'$ to 1.3$'$, we can remove local assumptions and show parts \ref{prop1-2}, \ref{prop1-3}, \ref{prop1-4} in general. Finally, part \ref{prop1-5} is obtained by the proof of Claim 2 with ``minimality of $H$" replaced by Lemma \ref{lem: strong-move sufficiency}.
\end{proof}

As a result, the original and created vertices are separated by degrees. Moreover, in such cases, the family $\mathcal{F}(G)$ can be generated by $\Delta$-Y moves alone without using any Y-$\Delta$ moves. This is Theorem \ref{thm: forward-move sufficiency}.

\begin{proofof}{Theorem \ref{thm: forward-move sufficiency}}
Suppose otherwise that $\mathcal F(G)\setminus \mathcal F'(G)$ is non-empty. For each $H \in \mathcal F(G)\setminus \mathcal F'(G)$, we define $y(H)$ to be the least number of Y-$\Delta$ moves needed to obtain $H$ from $G$ via a sequence of $\Delta$-Y or $Y$-$\Delta$ moves. Then, we have that 
$$
\min_{H \in \mathcal F(G)\setminus \mathcal F'(G)} y(H) = 1.
$$
Indeed, let $H \in \mathcal F(G)\setminus \mathcal F'(G)$ attain this minimum. Then, there is a chain $G=H_0 \xrightharpoonup{M_0} \dots \xrightharpoonup{M_{n-1}} H_n = H$ of $\Delta$-Y or Y-$\Delta$ moves such that there are $y(H)$ Y-$\Delta$ moves among $\{M_0, \dots,M_{n-1}\}$. Suppose $M_i$ is the first Y-$\Delta$ move, then we must have $H_{i+1} \in \mathcal F(G)\setminus \mathcal F'(G)$, otherwise we may replace the subchain $G=H_0 \xrightharpoonup{M_0} \dots \xrightharpoonup{M_{i}} H_{i+1}$ with a sequence of $\Delta$-Y moves, yielding a chain to $H$ with a smaller number of $Y$-$\Delta$ moves than $y(H)$, thus giving a contradiction. Then, since $y(H_{i+1}) \le 1$, we must have $\min_{H \in \mathcal F(G)\setminus \mathcal F'(G)} y(H) \le 1$, and hence the number is exactly 1 by definition.

Now, pick some $H \in \mathcal F(G)\setminus \mathcal F'(G)$ with $y(H) = 1$. Let $G=H_0 \xrightharpoonup{M_0} \dots \xrightharpoonup{M_{n-1}} H_n = H$ be a chain of the least number of Y-$\Delta$ moves from $G$ to $H$. Then, without loss of generality, by the argument above, we may assume that the only Y-$\Delta$ move is $M_{n-1}$. Since the only vertices of degree 3 are created vertices in $H_{n-1}$, there exists a created vertex $c \in H_{n-1}$ such that $M_{n-1} = \iota_{c}$. 

Since all moves $M_0,\dots,M_{n-2}$ are $\Delta$-Y moves, there exists a unique move $M_i$ in which $c$ is created. For $j = 0,\dots,n-2$, we denote $M_j = \iota_{T_j}$, where $T_j$ is the triangle in $H_j$ on which $M_j$ is applied. By part \ref{prop1-5} of Proposition \ref{prop:vertex-type_separation}, the neighbourhood of $c$ is invariant under $M_{i+1},\dots,M_{n-2}$. Moreover, since the neighbourhood $N(c)$ of $c$ is independent, we have that $|T_j \cap N(c)| \le 1$ for every $j = i+1,\dots, n-2$. This shows that $M_i$ commutes with $M_{i+1},\dots, M_{n-2}$. Therefore, $H$ is obtained from the sequence $M_0,\dots,M_{i-1}, M_{i+1},\dots,M_{n-2},M_{i},M_{n-1}$. Since $M_i$ and $M_{n-1}$ are inverses to each other, $H$ is obtained from a sequence of $\Delta$-Y moves $M_0,\dots,M_{i-1}, M_{i+1},\dots,M_{n-2}$, contradicting that $H \not \in \mathcal F'(G)$.

\end{proofof}

Finally, we remove local assumptions and record the following lemma.

\begin{lemma}[triangle blocks]
\label{lem: triangle blocks}
Let $G$ be a graph with $\delta(G)\ge 7$, and let $H\in\mathcal F(G)$. Then, there is a sequence of distinct edge-disjoint triangles $T_1, \dots, T_m$ in $G$ such that $H$ is obtained from $G$ via $\Delta$-Y moves on $T_1, \dots, T_m$. Moreover, the sequence is unique up to reordering.
\end{lemma}

\begin{proof}
We induct on the number of moves needed to reach $H$ from $G$ via $\Delta$-Y moves. The base case $H = G$ is trivial. Now, suppose $H \neq G$. Let $$G=H_0 \xrightharpoonup{\iota_{T_1}} \dots \xrightharpoonup{\iota_{T_m}} H_m = H$$ be a shortest chain of $\Delta$-Y moves from $G$ to $H$. Since $H_{m-1}$ can be reached by a shorter chain, we apply the induction hypothesis and obtain that $T_1,\dots,T_{m-1}$ is the unique sequence of distinct edge-disjoint triangles in $G$ to reach $H_{m-1}$ by $\Delta$-Y moves. $T_m \subset H_{m-1}$ is clearly distinct from $T_1,\dots,T_{m-1}$ because the vertices of a triangle form an independent set after a $\Delta$-Y move, thus does not admit a further move on any of $T_1,\dots,T_{m-1}$. For the same reason, $T_m$ is edge-disjoint from each of $T_1,\dots, T_{m-1}$.

Thus, no $\Delta$-Y move in $T_1,\dots,T_{m-1}$ modifies any edge of $T_m$ in $H_{m-1}$. Since no created vertex is contained in $T_m$ by part \ref{prop1-3} of Proposition \ref{prop:vertex-type_separation}, we have that the vertices of $T_m$ are original vertices. As a result, $T_m$ is a triangle in $G$.

Thus, $\iota_{T_m}$ commutes with any $\iota_{T_i}$ with $i<m$. 

Finally, uniqueness follows from 
$$
\{T_1,\dots,T_m\} = \{N(c): c\in H \text{ is a created vertex}\}.
$$
\end{proof}

\section{Complete Graphs and Partial Steiner Systems}

\begin{definition}[graph of partial Steiner systems]
\label{def:graph of pss}
For each $n\ge 3$, we define $\mathcal P(n)$ to be a directed graph of vertices $V(\mathcal P(n)) := \{\text{partial Steiner systems on }n \text{ points}\}$, and a forward edge is assigned from $\mathcal A \in V(\mathcal P(n))$ to $\mathcal A' \in V(\mathcal P(n))$ if $\mathcal A \subset \mathcal A'$ and that $|\mathcal A' \setminus \mathcal A| =  1$. That is to say, $\mathcal A'$ can be obtained from $\mathcal A$ by adjoining a new block.

Then, $\mathcal P(n)$ is naturally graded by rank $r(\mathcal A):= |\mathcal A|$.
\end{definition}

\begin{proofof}{Theorem \ref{thm: partial Steiner system identification}}
We define the grading by $r_{\mathcal G}(H): =  |V(H)| - n$. By Lemma \ref{lem: triangle blocks}, every $H\in\mathcal{G}(K_n)$ is obtained from
$K_n$ by $\Delta$-Y moves on a collection
\[
\mathcal A_H=\{T_1,\dots,T_m\}
\]
of distinct edge-disjoint triangles, unique up to ordering. By
Proposition \ref{prop:vertex-type_separation} part \ref{prop1-4}, the created vertices are in
bijection with their neighbourhoods $T_i$. Hence
\[
\Phi_n(H)=\mathcal A_H\in\mathcal P(n),
\]
so $\Phi_n$ is well-defined.

Conversely, for $\mathcal A\in\mathcal P(n)$, define $\Psi_n(\mathcal A)$ by
\[
V(\Psi_n(\mathcal A))
=[n]\sqcup\{c_T:T\in\mathcal A\},
\]
where $N(c_T)=T$, and where $x,y\in[n]$ are adjacent precisely when
$\{x,y\}$ is contained in no block of $\mathcal A$. Since the blocks of
$\mathcal A$ are edge-disjoint, they may be replaced by $\Delta$-Y moves
in any order, and therefore $\Psi_n(\mathcal A)\in\mathcal{G}(K_n)$.

It follows immediately from the definitions that
\[
\Phi_n(\Psi_n(\mathcal A))=\mathcal A.
\]
Conversely, if $\Phi_n(H)=\mathcal A_H$, then the description above shows
that the created vertices and all original--original edges of $H$ are
exactly those prescribed in $\Psi_n(\mathcal A_H)$. Thus
\[
\Psi_n(\Phi_n(H))=H.
\]
Hence $\Phi_n$ is a bijection.

Moreover, a $\Delta$-Y move on a triangle $T$ corresponds exactly to
adjoining the block $T$, so $\Phi_n$ is an isomorphism of directed graphs.
If $|\mathcal A_H|=m$, then
\[
r_{\mathcal G}(H)=|V(H)|-n=m=|\mathcal A_H|
=r_{\mathcal P}(\Phi_n(H)),
\]
so it preserves the grading. Finally, for every $\sigma\in S_n$,
\[
\Phi_n(\sigma H)
=\{\sigma N(c):c\text{ created}\}
=\sigma\Phi_n(H),
\]
and hence $\Phi_n$ is $S_n$-equivariant.
\end{proofof}

\section*{Acknowledgement}
The author thanks his supervisor, Matthew de Courcy-Ireland, for inspiring this work through a discussion of the Colin de Verdière parameter and topological graph embeddings, and for his careful reading of the manuscript and helpful advice.


\addcontentsline{toc}{section}{References}
\bibliographystyle{alpha}
\bibliography{ref}
\end{document}